\documentclass[12pt,reqno]{amsart}

\usepackage[a4paper,margin=1in]{geometry} 
\newcommand\bigw{\scalebox{.85}[0.85]{$\bigwedge$}}
\allowdisplaybreaks

\usepackage{amscd}
\usepackage{amssymb}
\usepackage{mathrsfs}
\usepackage{amsmath}
\usepackage{mathabx}
\usepackage{amsthm}
\usepackage[all,cmtip]{xy}
\usepackage{enumerate}
\usepackage{enumitem}
\usepackage{ytableau}
\usepackage{tikz}
\usepackage{tikz-cd} 
\usepackage{mathtools}
\usepackage{color}
\usepackage{wasysym}

\usepackage[plainpages,urlcolor=red,linktocpage=true]{hyperref}

\numberwithin{equation}{section}

\newcommand{\thmref}[1]{Theorem~\ref{#1}}
\newcommand{\secref}[1]{Section~\ref{#1}}

\newcommand{\lemref}[1]{Lemma~\ref{#1}}
\newcommand{\propref}[1]{Proposition~\ref{#1}}

\newcommand{\eqnref}[1]{(\ref{#1})}

\newcommand{\rmkref}[1]{Remark~\ref{#1}}

\newtheorem{theorem}{Theorem}[section]
\newtheorem{lemma}[theorem]{Lemma}
\newtheorem{proposition}[theorem]{Proposition}

\theoremstyle{definition}

\theoremstyle{remark}
\newtheorem{remark}[theorem]{Remark}

\newcommand{\RNum}[1]{\uppercase\expandafter{\romannumeral #1\relax}}

\begin{document}

\title[Unitary branching rules for $\mathfrak{gl}_{m|n}$]{Unitary branching rules for \\ the general linear Lie superalgebra}
\author{Mark Gould and Yang Zhang}
\address{School of Mathematics and Physics, The University of Queensland, St Lucia, QLD 4067, Australia}

\email{m.gould1@uq.edu.au}
\email{yang.zhang@uq.edu.au}
\begin{abstract}
 In terms of highest weights, we establish branching rules  for finite dimensional unitary simple modules of the general linear Lie superalgebra $\mathfrak{gl}_{m|n}$. Our proof uses the Howe duality for $\mathfrak{gl}_{m|n}$, as well as  branching rules for Kac modules. Moreover, we derive the branching rules of  type 2 unitary simple $\mathfrak{gl}_{m|n}$-modules, which are dual to the aforementioned unitary modules.
\end{abstract}


\maketitle

\section{Introduction}

Branching rules describe how a module of an algebra decomposes when restricted to a subalgebra. These rules are fundamental  in the  representation theory of Lie algebras, leading to important constructions such as Gelfand-Tsetlin (GT) bases, which are indispensable for analysing  the structure of simple modules \cite{GW09, Mo06}. Recently, there has been substantial interest in generalising branching rules to Lie superalgebras \cite{SV10,Mo11,GIW13,CPT15,LZ24}. However, this task is far more challenging than in the classical case, primarily due to the generic non-semisimplicity of representations of Lie superalgebras.  Most of the existing results focus on polynomial modules of the general linear superalgebra, where certain techniques, such as the Howe duality \cite{CLZ04,CW01,CPT15,LZ24}  and Young diagrams \cite{BR87,Mo11}, have been applied to describe the module decompositions. 

In this note, we are concerned with finite-dimensional unitary $\mathfrak{gl}_{m|n}$-modules, which form a semisimple category. Unitary modules for Lie superalgebras were first introduced in \cite{SNR77} as a generalisation of Hermitian representations of Lie algebras. It was also shown in \cite{SNR77} that  there are two  types of unitary $\mathfrak{gl}_{m|n}$-modules, distinguished by the  star operations on $\mathfrak{gl}_{m|n}$.  For brevity, we refer to type 1 unitary modules simply as unitary modules.  Subsequently, Gould and Zhang \cite{GZ90} classified both types of  unitary $\mathfrak{gl}_{m|n}$-modules, which include all  polynomial modules  and  certain typical non-polynomial modules. 

Our main result (\thmref{thm: main}) provides a branching rule, expressed in terms of highest weights, that describes the decomposition of a simple unitary $\mathfrak{gl}_{m|n}$-module into simple  unitary $\mathfrak{gl}_{m|n-1}$-modules for all positive integers $m,n$. For typical  unitary modules, the proof relies on a more general branching rule for Kac modules. In the atypical case, we employ Howe duality to derive the unitary module decompositions. Moreover, we determine  branching rules (\thmref{thm: maintype2}) for type 2 unitary $\mathfrak{gl}_{m|n}$-modules, which are related to  the type 1 unitary modules via duality. 

Generally, for any finite dimensional simple $\mathfrak{gl}_{m|n}$-module, \cite{GIW13} provides some  necessary conditions for its decomposition into simple $\mathfrak{gl}_{m|n-1}$-modules.  In \cite{GIW14, GIW15}, these necessary conditions are applied to parameterise the basis vectors of a unitary simple $\mathfrak{gl}_{m|n}$-module in terms of  GT patterns. This yields explicit matrix element formulas that describe the action of $\mathfrak{gl}_{m|n}$ on the corresponding GT basis; see also \cite{SV10,Mo11} for the matrix element formulas for covariant representations.  This paper establishes a branching rule which is both necessary and sufficient for the decomposition of a finite-dimensional unitary simple $\mathfrak{gl}_{m|n}$-module of both types under the action of  the subalgebra $\mathfrak{gl}_{m|n-1}$. In particular, our branching rules address cases that were not carefully considered in \cite{GIW14, GIW15} for  unitary modules of both types; see \rmkref{rmk: comp}. Therefore, our branching rules, expressed  in terms of highest weights,  provide precise GT patterns (cf. \cite{SV10,GIW14,GIW15}) for  unitary modules of both types. 

This paper is organised as follows. Section \ref{sec: pre} covers the basics of $\mathfrak{gl}_{m|n}$,  including the Howe duality for $\mathfrak{gl}_{m|n}$,  and recalls branching rules and Pieri's rule for the general linear Lie algebra. Section \ref{sec: mainres} presents  the main result, and divides the proof of the branching rules for type 1 unitary modules into typical and atypical cases. Section \ref{sec: type2uni} derives the branching rules for type 2 unitary modules.

\vspace*{0.2cm}
\noindent {\bf Notation.} Throughout this paper, we work over the complex field $\mathbb{C}$, unless otherwise stated. We denote by $\mathbb{Z}$ the set of integers,  by   $\mathbb{Z}_+$ the set of non-negative integers,  and by $\mathbb{Z}_2=\{\bar{0}, \bar{1}\}$ the set of integers modulo 2. For any vector space $V$, we denote $V^*:={\rm Hom}_{\mathbb{C}}(V, \mathbb{C})$, the dual space of $V$. For any Lie superalgebra $\mathfrak{g}$, we denote by ${\rm U}(\mathfrak{g})$ its universal enveloping algebra. For any superspace $V=V_{\bar{0}}\oplus V_{\bar{1}}$, define the parity functor $[\ ]: V\rightarrow \mathbb{Z}_2$  by $[v]=i$ if $v\in V_i$ for $i\in \mathbb{Z}_2$. For any positive integers $a,b$, let $\delta_{ab}$ denote the Kronecker delta function. We write $\otimes$ for $\otimes_{\mathbb{C}}$.

\section{Preliminaries}\label{sec: pre}

\subsection{The general linear Lie superalgebra}
Let $\mathbb{C}^{m|n}$ denote the complex superspace of dimension $m|n$ with the standard homogeneous  basis $e_{i}, 1\leq i\leq m+n$, such that $\{e_1, \dots e_m\}$ is a basis for  the even subspace $\mathbb{C}^{m|0}$ and $\{e_{m+1}, \dots e_{m+n}\}$ is a basis for  the odd subspace $\mathbb{C}^{0|n}$. As a complex superspace, the \emph{general Lie superalgebra} $\mathfrak{gl}_{m|n}$ consists of $(m+n)\times (m+ n)$-block matrices  with respect to the standard  homogeneous basis, i.e., 
\[ \mathfrak{gl}_{m|n}:= \Big\{ \begin{pmatrix}
 A & B\\ 
 C & D
\end{pmatrix} \mid  A\in \mathfrak{M}_{m,m}, B\in \mathfrak{M}_{m, n}, C\in \mathfrak{M}_{n,m}, D\in \mathfrak{M}_{n, n} \Big\}, 
\] 
 where $\mathfrak{M}_{p,q}$ denotes the complex space of $(p\times q)$-matrices for any positive integers $p,q$. The even subalgebra $(\mathfrak{gl}_{m|n})_{\bar{0}}=\mathfrak{gl}_m \oplus \mathfrak{gl}_n$ consists of matrices with $B$ and $C$ being zero blocks, while the odd subspace $(\mathfrak{gl}_{m|n})_{\bar{1}}$ consists of matrices with $A$ and $D$ being zero blocks. The Lie bracket is defined by $[X,Y]= XY-(-1)^{[X][Y]}YX$ for any homogeneous $X,Y\in \mathfrak{gl}_{m|n}$.

For any pair $1\leq i,j\leq m+n$, denote by $E_{ij}$ the matrix unit of $\mathfrak{gl}_{m|n}$ such that $E_{ij}e_k= \delta_{jk}e_i$ for all $1\leq k\leq m+n$. Then $\{E_{ij}\mid 1\leq i,j\leq m+n\}$ is a linear basis of $\mathfrak{gl}_{m|n}$. The \emph{Cartan subalgebra} $\mathfrak{h}_{m|n}$ is spanned by  $\{E_{ii}\mid 1\leq i\leq m+n \}$, while the \emph{standard Borel subalgebra} $\mathfrak{b}_{m|n}$ containing $\mathfrak{h}_{m|n}$ is spanned by  $\{E_{ij}\mid 1\leq i\leq j\leq m+n \}$.

Let $\{\epsilon_i \mid 1 \leq i \leq m+n\}$ be the dual basis of $\mathfrak{h}_{m|n}^*$ such that $\epsilon_i(E_{jj}) = \delta_{ij}$ for $1 \leq i,j \leq m+n$. The set of positive roots relative to $\mathfrak{b}_{m|n}$ is given by $\Delta^+ = \{\epsilon_i - \epsilon_j \mid 1 \leq i < j \leq m+n\}$. For convenience, we write $\delta_{\mu}= \epsilon_{m+\mu}$ for $1\leq \mu \leq n$. The sets of even and odd positive roots are denoted by $\Delta^+_{\bar{0}}$ and $\Delta^+_{\bar{1}}$, respectively, where
\begin{align*}
  \Delta^+_{\bar{0}} &= \{\epsilon_i - \epsilon_j \mid 1 \leq i < j \leq m\} \cup \{\delta_{\mu} - \delta_{\nu} \mid 1 \leq \mu < \nu \leq n\},\\
  \Delta^+_{\bar{1}} &= \{\epsilon_i - \delta_{\mu} \mid 1 \leq i \leq m, 1\leq \mu\leq n\}.
\end{align*}
 Let $(-,-): \mathfrak{h}^*_{m|n}\times \mathfrak{h}^*_{m|n}\rightarrow \mathbb{C}$ be the symmetric bilinear form defined by
\[(\epsilon_i, \epsilon_j)=\delta_{ij}, \quad (\delta_{\mu}, \delta_{\nu})=-\delta_{\mu,\nu}, \quad (\epsilon_i, \delta_{\mu})=0, \]
for $1\leq i,j \leq m$ and  $1\leq \mu,\nu\leq n$. We denote by $\rho=\rho_{m|n}$ the half graded sum of positive roots, and more explicitly,   
\[ \rho=\frac{1}{2}\sum_{i=1}^m (m-n-2i+1)\epsilon_i+ \frac{1}{2}\sum_{\mu=1}^n(m+n-2\mu+1)\delta_{\mu}.   \]

Every finite dimensional simple $\mathfrak{gl}_{m|n}$-module $L(\Lambda)=L^{m|n}(\Lambda)$ is uniquely characterised by the highest weight 
\[\Lambda=(\lambda_1,\dots,  \lambda_m, \omega_1, \dots, \omega_n):=\sum_{i=1}^m \lambda_i\epsilon_i + \sum_{\mu=1}^n\omega_{\mu} \delta_{\mu}\in \mathfrak{h}_{m|n}^*, \]
where $\lambda_i-\lambda_{i+1}\in \mathbb{Z}_+$ and $\omega_\mu-\omega_{\mu+1}\in \mathbb{Z}_+$ for all relevant $i,\mu$, that is, $\Lambda$ is a dominant integral weight of the even subalgebra $\mathfrak{gl}_m \oplus \mathfrak{gl}_n$. We denote by $P_+^{m|n}$ the set of all $\mathfrak{gl}_m \oplus \mathfrak{gl}_n$-dominant integral weights. Following Kac \cite{Kac77a}, we say that $\Lambda$ and the associated simple module $L(\Lambda)$ are \emph{typical} if 
\[\prod_{\alpha\in \Delta^+_{\bar{1}}}(\Lambda+\rho, \alpha)\neq 0; \]
otherwise, they are called \emph{atypical}.

Recall from \cite{Kac77a} that every finite dimensional simple $\mathfrak{gl}_{m|n}$-module $L(\Lambda)$ can be constructed as follows. Let $L_0(\Lambda)$ be the simple $\mathfrak{g}_{\bar{0}}$-module with highest weight $\Lambda\in P_+^{m|n}$. Note that $\mathfrak{g}=\mathfrak{gl}_{m|n}$ admits a $\mathbb{Z}$-grading:
\begin{equation}\label{eq: Zgrad}
  \mathfrak{g}= \mathfrak{g}_{-1}\oplus \mathfrak{g}_0 \oplus \mathfrak{g}_{1},
\end{equation}
where $\mathfrak{g}_0=\mathfrak{gl}_m \oplus \mathfrak{gl}_n$, and $\mathfrak{g}_{1}$ (resp. $\mathfrak{g}_{-1}$) is spanned by all $E_{ij}$ (resp. $E_{ji}$) with $1\leq i\leq m<j \leq m+n$. Then $L_0(\Lambda)$ can be extended as a $\mathfrak{g}_0\oplus \mathfrak{g}_{1}$-module with  $\mathfrak{g}_1$ acting trivially. We define the Kac module $K(\Lambda)$ with highest weight $\Lambda\in P_+^{m|n}$ by 
\begin{equation}\label{eq: isoK}
	K(\Lambda):= {\rm U}(\mathfrak{g}) \bigotimes_{{\rm U}(\mathfrak{g}_{0}\oplus \mathfrak{g}_{1})} L_{0}(\Lambda)= {\rm U}(\mathfrak{g}_{-1})\otimes L_{0}(\Lambda),
\end{equation}
where the right equality follows from  the PBW theorem for ${\rm U}(\mathfrak{g})$. Note that ${\rm U}(\mathfrak{g}_{-1})$ is the exterior algebra $\bigw{(\mathfrak{g}_{-1})}$ of the vector space $\mathfrak{g}_{-1}$, which carries the structure of a $\mathfrak{g}_0$-module.  This equips  $K(\Lambda)$ with a natural $\mathbb{Z}$-grading $K(\Lambda)= \bigoplus_{k=0}^{mn} K(\Lambda)_k$, where $K(\Lambda)_k=\bigw^{k}{\mathfrak{g}_{-1}}\otimes L_0(\Lambda)$ is a tensor product of two $\mathfrak{g}_0$-modules. 

Recall from \cite{Kac77b} that   $K(\Lambda)$ is a simple $\mathfrak{gl}_{m|n}$-module if and only if $\Lambda$ is typical. If $\Lambda$ is atypical, then $K(\Lambda)$ is not simple and contains a unique proper maximal submodule $M(\Lambda)$. Then $L(\Lambda)$ arises as the unique simple quotient  $K(\Lambda)/M(\Lambda)$. It follows the construction that $L(\Lambda)$ has a natural $\mathbb{Z}$-grading:
\begin{equation}\label{eq: Lgrad}
  L(\Lambda)= \bigoplus_{k=0}^{d_{\Lambda}} L(\Lambda)_k,
\end{equation} 
where $0\leq d_{\Lambda}\leq mn$, and  each summand $L(\Lambda)_k$ is a $\mathfrak{g}_0$-module. In particular, $L(\Lambda)_0=L_0(\Lambda)$ is a simple  $\mathfrak{g}_0$-module.

\subsection{Unitary modules}

We recall some basic facts about $\ast$-algebras and their unitary modules \cite{SNR77,GZ90,CLZ04}. A $\ast$-superalgebra is an associative superalgebra $A=A_{\bar{0}}\oplus A_{\bar{1}}$ equipped with an even anti-linear anti-involution $\phi: A\rightarrow A$, i.e., 
\[ 
\begin{aligned}
 \phi(ca)&=\bar{c}a, \quad c\in \mathbb{C}, a\in A, \\
 \phi(ab)&= \phi(b)\phi(a), \quad a,b\in A, 
\end{aligned}
\]
where $\bar{c}$ denote the complex conjugate of $c$. A $\ast$-superalgebra homomorphism $f: (A, \phi)\rightarrow (A', \phi')$ is a superalgebra homomorphism satisfying $f\circ \phi= \phi'\circ f$.  

Let $(A,\phi)$ be a $\ast$-superalgebra and let $V$ be a $\mathbb{Z}_2$-graded $A$-module. A Hermitian form $\langle -,- \rangle$ on $V$ is said to be \emph{positive definite} if $\langle v, v \rangle > 0$ for all $v \neq 0$, and \emph{contravariant} if $\langle av, w \rangle = \langle v, \phi(a)w \rangle$ for all $a \in A$ and $v, w \in V$. An $A$-module equipped with a positive definite contravariant Hermitian form is called a \emph{unitary $A$-module}.  

As an example, the following anti-linear anti-involution defines a  $\ast$-structure of  ${\rm U}(\mathfrak{gl}_{m|n})$:
\[ \phi(E_{ij})= E_{ji}, \quad 1\leq i,j\leq m+n. \]
A simple $\mathfrak{gl}_{m|n}$-module $L(\Lambda)$ is said to be (type 1) \emph{unitary}\footnote{This is called a  type 1 star module in \cite{GZ90}. Type 2 star (unitary) modules will be discussed separately in \secref{sec: type2uni}  } if there exists a positive definite Hermitian form $\langle-, -\rangle$ on $L(\Lambda)$ such that for all $v,w\in  L(\Lambda)$, 
\begin{equation}\label{eq: Her}
  \langle E_{ij}v, w\rangle= \langle v, E_{ji} w\rangle, \quad 1\leq i,j \leq m+n.
\end{equation}
In particular, when $n=0$, these definitions apply directly to the Lie algebra $\mathfrak{gl}_m$. It is well known that a simple $\mathfrak{gl}_m$-module $L(\Lambda)$ is unitary if and only if $\Lambda$ is real and dominant integral.  

Clearly,  any positive definite contravariant Hermitian form on $L(\Lambda)$ induces  a  unitary $\mathfrak{g}_0$-module structure on $L(\Lambda)_0$ by restriction.  Thus, a necessary condition for a $\mathfrak{gl}_{m|n}$-module $L(\Lambda)$ to be unitary  is that $\Lambda$ is a real $\mathfrak{g}_0$-dominant integral weight. 

We denote by $\Lambda\in D_+^{m|n}$ the set of   real $\mathfrak{g}_0$-dominant integral weights. 

\begin{theorem}\cite{GZ90}\label{thm: unitary}
Let $\Lambda\in D_+^{m|n}$. The finite-dimensional simple  $\mathfrak{gl}_{m|n}$-module $L(\Lambda)$ is  unitary if and only if 
\begin{enumerate}
\item $(\Lambda+ \rho, \epsilon_m-\delta_n)>0 $; or 
\item there exists $\mu\in \{1, 2, \dots, n\}$ such that 
\[(\Lambda + \rho, \epsilon_m-\delta_{\mu}) = (\Lambda, \delta_{\mu}-\delta_n)=0. \]
\end{enumerate}
\end{theorem}

\begin{remark}\label{rmk: typ}
  Note that if $(\Lambda+\rho, \epsilon_m-\delta_{n})>0$, then for any $1\leq i\leq m$ and $1\leq \mu \leq n$, 
 		\[
\begin{aligned}
 		 (\Lambda+ \rho,\epsilon_i-\delta_{\mu}) &= (\Lambda+\rho,\epsilon_i-\epsilon_m + \epsilon_m-\delta_n+ \delta_n-\delta_{\mu})\\
 		 &=(\lambda_i-\lambda_m + m-i)+(\Lambda+\rho, \epsilon_m-\delta_{n})+ (\omega_{\mu}-\omega_n + n-\mu)>0.
\end{aligned}
 \]
 It follows that any unitary module $L(\Lambda)$ with $\Lambda\in D^{m|n}_+$ is either typical or atypical, depending on whether condition (1) or (2) is satisfied in \thmref{thm: unitary}.
\end{remark}

\subsection{The \texorpdfstring{$(\mathfrak{gl}_d, \mathfrak{gl}_{m|n})$}{(gld,glmn)}-Howe duality}
Fix a  positive integer $d$, and let $S(\mathbb{C}^d\otimes \mathbb{C}^{m|n})$ denote the supersymmetric algebra of $\mathbb{C}^d\otimes \mathbb{C}^{m|n}$. The natural action of $\mathfrak{gl}_d \oplus \mathfrak{gl}_{m|n}$ on $\mathbb{C}^d\otimes \mathbb{C}^{m|n}$ extends to $S(\mathbb{C}^d\otimes \mathbb{C}^{m|n})$. We will recall the $(\mathfrak{gl}_d, \mathfrak{gl}_{m|n})$-Howe duality, which yields a multiplicity-free decomposition of $S(\mathbb{C}^d\otimes \mathbb{C}^{m|n})$ into $\mathfrak{gl}_d \oplus \mathfrak{gl}_{m|n}$-simple modules. In this way, we realise a class of  simple unitary $\mathfrak{gl}_{m|n}$-modules within $S(\mathbb{C}^d\otimes \mathbb{C}^{m|n})$ whose weights are integral.

Before proceeding, we recall some combinatorial definitions from \cite{BR87,GW09,CW12}. A finite non-increasing sequence $\lambda=(\lambda_1, \dots, \lambda_d)$ of non-negative integers is called a \emph{partition} of length $d$. Denote by $\mathcal{P}_d$ the set of all partitions of length $d$. A \emph{Young diagram} associated with $\lambda$ is a collection of left-justified rows of boxes, where the $i$-th row contains $\lambda_i$ boxes.  The conjugate partition of $\lambda$ is $\lambda^c=(\lambda^c_1, \dots, \lambda^c_{\ell})$, where $\ell=\lambda_1$ and $\lambda^c_i=\# \{j \mid \lambda_j\geq i \}$ for $1\leq i\leq \ell$. 

An \emph{$(m,n)$-hook partition} is a partition $\lambda=(\lambda_1, \dots, \lambda_{d})$ for which $\lambda_{m+1}\leq n$, where we set $\lambda_{m+1}=0$ if $m\geq d$. Associated to such $\lambda$, we define
\begin{equation}\label{eq: lanat}
  \lambda^{\natural}:=(\lambda_1, \dots, \lambda_m, \langle \lambda^c_1-m\rangle, \dots, \langle \lambda^c_n-m\rangle),
\end{equation}
where $\langle \lambda^c_i-m\rangle:={\rm max}\{\lambda^c_i-m, 0\}$ for $1\leq i\leq n$.
Let $\mathcal{P}_{m|n}$ denote the set of all $(m,n)$-hook partitions. We have the following Howe duality \cite{How89, Ser01}.

\begin{theorem}[$(\mathfrak{gl}_d, \mathfrak{gl}_{m|n})$-Howe duality]\label{thm: Howe}
  As a $\mathfrak{gl}_d\oplus\mathfrak{gl}_{m|n}$-module,  $S(\mathbb{C}^d\otimes \mathbb{C}^{m|n})$ has the following multiplicity-free decomposition
  \[ S(\mathbb{C}^d\otimes \mathbb{C}^{m|n})\cong \bigoplus_{\lambda \in \mathcal{P}_d \cap \mathcal{P}_{m|n}} L^d(\lambda) \otimes L^{m|n}(\lambda^{\natural}),  \]
  where $L^{d}(\lambda)$ and $L^{m|n}(\lambda^{\natural})$ are simple $\mathfrak{gl}_d$- and $\mathfrak{gl}_{m|n}$-modules with highest weights $\lambda$ and $\lambda^{\natural}$, respectively.
\end{theorem}

We proceed next to recall that $S(\mathbb{C}^d\otimes \mathbb{C}^{m|n})$ is a unitary $\mathfrak{gl}_d\oplus\mathfrak{gl}_{m|n}$-module, as detailed in  \cite{CLZ04,CW01}. 
Let $f_a, 1\leq a\leq d,$ be the standard basis of $\mathbb{C}^{d}$ with $[f_a]=\bar{0}$ for all $a$.  Then $S(\mathbb{C}^d\otimes \mathbb{C}^{m|n})$ is isomorphic to the polynomial superalgebra $\mathbb{C}[\mathbf{x}]$ in variables $x_{ai}:= f_a\otimes e_i, 1\leq a\leq d, 1\leq i\leq m+n$.
Denote by  $\partial_{ai}:= \partial/\partial_{x_{ai}} $ the partial derivatives,  and  note that $[\partial_{ai}]=[x_{ai}]=[f_a]+[e_i]$ for all $a,i$.  
 
For $1\leq a, b \leq d$,  denote by $e_{ab}$ the matrix units of $\mathfrak{gl}_d$ relative to the standard basis $f_a, 1\leq a\leq d,$ such that $e_{ab}f_c= \delta_{bc} f_a$ for all $1\leq a,b,c\leq d$.
 The linear representation $\rho: {\rm U}(\mathfrak{gl}_d\oplus \mathfrak{gl}_{m|n}) \rightarrow {\rm End}_{\mathbb{C}}(\mathbb{C}[\mathbf{x}])$ is defined by
\begin{align*}
  \rho(e_{ab})= \sum_{i=1}^{m+n} x_{ai} \partial_{bi}, \quad 
  \rho(E_{ij})= \sum_{a=1}^d x_{ai}\partial_{aj},  
\end{align*} 
for  $1\leq a,b\leq d$ and $1\leq i,j \leq m+n$.  
It is straightforward to verify that the differential operators $\rho(e_{ab})$ (resp. $\rho(E_{ij})$) satisfy the commutation relations of $\mathfrak{gl}_d$ (resp. $\mathfrak{gl}_{m|n}$), and $\rho(e_{ab})$ commutes with $\rho(E_{ij})$ for all   $a,b,i,j$. 
Moreover, let $D[\mathbf{x}]$ denote the oscillator superalgebra generated by $x_{ai}$ and $\partial_{ai}$ for $1\leq a\leq d$ and $1\leq i\leq m+n$. Then $\rho$ extends to an associative superalgebra homomorphism
$\rho: {\rm U}(\mathfrak{gl}_d\oplus \mathfrak{gl}_{m|n}) \rightarrow D[\mathbf{x}].$  
Note that $\mathbb{C}[\mathbf{x}]$ is a simple $D[\mathbf{x}]$-module.

We now define a Hermitian form on $\mathbb{C}[\mathbf{x}]$ induced from  the $\ast$-structure on $D[\mathbf{x}]$, thus making $\mathbb{C}[\mathbf{x}]$ into a unitary $\mathfrak{gl}_d\oplus \mathfrak{gl}_{m|n}$-module. The oscillator superalgebra $D[\mathbf{x}]$ is a $\ast$-superalgebra with an anti-linear anti-involution $\psi$ defined by 
\[ \psi(x_{ai})= \partial_{ai}, \quad \psi(\partial_{ai})= x_{ai}, \quad  1\leq a\leq d, 1\leq i\leq m+n.   \]
This induces a unique Hermitian form $\langle -, -\rangle$ on $\mathbb{C}[\mathbf{x}]$ such that $\langle 1, 1\rangle=1$, and 
\[ \langle fg,h \rangle= \langle g, \psi(f) h \rangle, \quad f,g,h \in \mathbb{C}[\mathbf{x}]. \]
By definition, $\langle f,f \rangle>0$ for any monomial $f\in \mathbb{C}[\mathbf{x}]$ and hence $\langle -, -\rangle$ is positive definite. 

Similarly, ${\rm U}(\mathfrak{gl}_d\oplus \mathfrak{gl}_{m|n})$ is a $\ast$-superalgebra with an anti-linear anti-involution $\phi$ defined by 
\[ \phi(e_{ab})=e_{ba}, \quad \phi(E_{ij})=E_{ji}, \quad 1\leq a,b \leq d, 1\leq i,j\leq m+n. \]
It is straightforward to check that $\rho \phi (X)= \psi \rho(X)$ for all $X\in \mathfrak{gl}_d\oplus \mathfrak{gl}_{m|n}$, hence these two $*$-structures are compatible. We obtain the following \cite{CLZ04,CW01}. 

\begin{theorem}\label{thm: polyunitary}
  The polynomial superalgebra $\mathbb{C}[\mathbf{x}]\cong S(\mathbb{C}^d\otimes \mathbb{C}^{m|n})$ is a unitary ${\rm U}(\mathfrak{gl}_d\oplus \mathfrak{gl}_{m|n})$-module. In particular, every simple ${\rm U}(\mathfrak{gl}_d\oplus \mathfrak{gl}_{m|n})$-submodule of $\mathbb{C}[\mathbf{x}]$ is unitary. 
\end{theorem}

It follows from  \thmref{thm: Howe} and \thmref{thm: polyunitary} that every finite dimensional simple $\mathfrak{gl}_{m|n}$-module $L^{m|n}(\lambda^{\natural})$ associated to an $(m,n)$-hook partition $\lambda$ is unitary. We refer to $L^{m|n}(\lambda^{\natural})$ as  a \emph{polynomial module}, whose highest weight $\lambda^{\natural}$ is integral.

\subsection{Branching rules and Pieri's rule for \texorpdfstring{$\mathfrak{gl}_m$}{glm}}
We recall the branching rules and Pieri's rule  for  $\mathfrak{gl}_m= \mathfrak{gl}(\mathbb{C}^m)$ and  refer to \cite{GW09} for more details. Let $P^{m}_+$ denote the set of dominant integral weights of $\mathfrak{gl}_m$. 

For any  $\Lambda=(\lambda_1, \dots, \lambda_{m})\in P^{m}_+$  and $\Lambda'=(\lambda'_1, \dots, \lambda'_{m-1})\in P^{m-1}_+$,  we say that $\Lambda'$ interlaces $\Lambda$, and  write $\Lambda\downarrow \Lambda'$ if 
\begin{equation}\label{eq: ladown}
	\lambda_1\geq \lambda'_1\geq \lambda _2 \geq \lambda'_2\geq \cdots \geq \lambda'_{m-1}\geq \lambda_{m},  \quad \lambda_i-\lambda'_i\in \mathbb{Z}_+ \text{\ for\ } 1\leq i\leq m-1. 
\end{equation}

\begin{theorem}[Branching rules]\label{thm: glbr}
 Let $\Lambda\in P_+^{m}$. 
 The finite-dimensional  simple $\mathfrak{gl}_{m}$-module $L^{m}(\Lambda)$ decomposes into a multiplicity-free direct sum of simple  $\mathfrak{gl}_{m-1}$-modules
 \[L^{m}(\Lambda)\cong \bigoplus_{\Lambda\downarrow \Lambda'} L^{m-1}(\Lambda').  \]
\end{theorem}

For any $\Lambda=(\lambda_1, \dots, \lambda_{m})$ and $\Lambda'=(\lambda'_1, \dots, \lambda'_{m})\in P^m_{+}$, we shall write $\Lambda \rightarrow \Lambda'$ if 
\begin{equation}\label{eq: lasim}
	\lambda_i-\lambda'_i\in \{0,1\}, \quad 1\leq i\leq m,
\end{equation}
and define the difference $|\Lambda-\Lambda'|:= \sum_{i=1}^m (\lambda_i-\lambda'_i)$.

\begin{theorem}[Pieri's rule]\label{thm: Pieri}
  Let $\Lambda\in P_+^{m}$. For each $k=0, \dots, m$, let  $\bigw^{k}(\mathbb{C}^m)$ and $\bigw^{k}((\mathbb{C}^m)^*)$ be the $k$-th exterior powers of the  natural $\mathfrak{gl}_m$-module $\mathbb{C}^m$ and its dual $(\mathbb{C}^m)^*$, respectively.  
  \begin{enumerate}
  \item As a $\mathfrak{gl}_m$-module, $\bigw^{k}(\mathbb{C}^m)\otimes L^m(\Lambda)\cong \bigoplus_{\Lambda' \rightarrow \Lambda, |\Lambda'-\Lambda|=k} L^{m}(\Lambda'). $
  \item As a $\mathfrak{gl}_m$-module, $\bigw^{k}((\mathbb{C}^m)^*)\otimes L^m(\Lambda)\cong \bigoplus_{\Lambda \rightarrow \Lambda',  |\Lambda-\Lambda'|=k} L^{m}(\Lambda'). $
  \end{enumerate}
\end{theorem}

\thmref{thm: Pieri} is also a consequence of \cite[Lemma A]{Gou89}.

\section{Unitary branching rules}

\subsection{Main result}\label{sec: mainres}
Throughout this section, $m$ and $n$ are positive integers. Let $\Lambda=(\lambda_1,\dots, \lambda_m, \omega_1, \dots, \omega_{n}) \in P^{m|n}_+$   and $\Lambda'= (\lambda'_1,\dots, \lambda'_m, \omega'_1, \dots, \omega'_{n-1})\in  P^{m|n-1}_+$. We say that $\Lambda'$ interlaces $\Lambda$,  and  write $\Lambda \downarrow \Lambda'$ if the  following interlacing conditions hold:
\begin{enumerate}
\item [(C1)]  For all $1\leq i\leq m-1$, $\lambda_i-\lambda'_i\in \{0,1\}$. For $i=m$: 
\begin{itemize}
  \item If $(\Lambda+\rho, \epsilon_m-\delta_1) = 0$, then $\lambda_m = \lambda'_m$;
  \item Otherwise, $\lambda_m - \lambda'_m \in \{0,1\}$.
  \end{itemize}
\item  [(C2)] $ \omega_{\mu} \geq \omega'_{\mu}\geq  \omega_{\mu+1}$ and $\omega_{\mu}- \omega'_{\mu}\in \mathbb{Z}_+$ for all  $1\leq \mu \leq n-1$. 
\end{enumerate}
If $n=1$, only (C1) is required. Note  that in (C1),  the equality $\lambda_m = \lambda'_m$ implies that  $\lambda_j= \lambda'_j$ for any $j<m$ such that $(\Lambda, \epsilon_j-\epsilon_m)=0$, since $\Lambda'$ is a dominant integral weight.

\begin{theorem}\label{thm: main}
Let  $L^{m|n}(\Lambda)$ be a finite-dimensional unitary simple $\mathfrak{gl}_{m|n}$-module with highest weight $\Lambda\in D^{m|n}_+$.  Then  $L^{m|n}(\Lambda)$ decomposes into a multiplicity-free direct sum of unitary simple $\mathfrak{gl}_{m|n-1}$-modules: 
\[
L^{m|n}(\Lambda)\cong \bigoplus_{\Lambda \downarrow \Lambda'} L^{m|n-1}(\Lambda').
\]
\end{theorem}
\begin{remark}\label{rmk: comp}
  Condition (C1) stipulates that $\lambda_m=\lambda'_m$ when $(\Lambda+\rho, \epsilon_m-\delta_1) = 0$ for any $\mathfrak{gl}_{m|n}$. However, the necessary conditions given in \cite[Theorem 9]{GIW14} specify this only for the case $\mathfrak{gl}_{m|1}$. Note that  if the highest weight $\Lambda$ is unitary, then by \cite[Proposition 4]{GZ90} $(\Lambda+\rho, \epsilon_m-\delta_1) = 0$ implies $(\Lambda, \delta_1-\delta_n)=0$. Hence this  distinction is particularly relevant for 1-dimensional $\mathfrak{gl}_{m|n}$-modules with weights of the form  $(s,\dots, s, -s, \dots, -s)$ (consisting of $m$ copies of $s\in \mathbb{R}$ and $n$ copies of $-s$), which are both atypical and unitary. Any such $1$-dimensional module restricts to a $1$-dimensional $\mathfrak{gl}_{m|n-1}$-module with the weight comprising $m$ copies of $s$ and $n-1$ copies of $-s$. 
\end{remark}

We split the proof of \thmref{thm: main} into the typical and atypical cases, which are explicitly addressed in \secref{sec: typical} and \secref{sec: atypical}, respectively.

\subsection{Proof for the typical case}\label{sec: typical}
Let $\mathfrak{g}'=\mathfrak{gl}_{m|n-1}$ denote  the subalgebra of $\mathfrak{g}=\mathfrak{gl}_{m|n}$ spanned by $E_{ij}$ for $1\leq i,j \leq m+n-1$. Recall from \eqref{eq: Zgrad} the $\mathbb{Z}$-grading of $\mathfrak{g}$, and similarly, let $\mathfrak{g}'= \mathfrak{g}'_{-1} \oplus \mathfrak{g}'_{0} \oplus \mathfrak{g}'_1$ be the  $\mathbb{Z}$-grading of $\mathfrak{g}'$. The vector space  $\mathfrak{g}_{-1}$ (resp. $\mathfrak{g}'_{-1}$) is linearly spanned by $E_{ji}$ for $1\leq i\leq m<j\leq m+n$  (resp. $1\leq i\leq m<j\leq m+n-1$).  

Note that $\mathfrak{g}_{-1}\cong (\mathbb{C}^m)^*\otimes \mathbb{C}^{n}$ as a $\mathfrak{g}_0$-module, while $\mathfrak{g}'_{-1}\cong (\mathbb{C}^m)^*\otimes \mathbb{C}^{n-1}$ as a $\mathfrak{g}'_0$-module. Restricting to $\mathfrak{g}'_0=\mathfrak{gl}_m \oplus \mathfrak{gl}_{n-1}$, we have the decomposition:
 \begin{equation}\label{eq: g1dec}
  \mathfrak{g}_{-1}|_{\mathfrak{g}'_0}= \mathfrak{g}'_{-1} \oplus \mathfrak{m}'_{-1},
 \end{equation}
where $\mathfrak{m}'_{-1}\cong (\mathbb{C}^m)^*$ as a $\mathfrak{gl}_m$-module with trivial $\mathfrak{gl}_{n-1}$-action  and is  linearly spanned by $E_{m+n,i}$ for $1\leq i \leq m$.

We begin with the branching rule for  Kac modules; see also \cite[Appendix A]{GIW14} for a character-theoretic proof. 
\begin{proposition}\label{prop: Kacdec}
Let $\Lambda\in P^{m|n}_+$ with $m\geq 1$ and $n>1$. The Kac module $K^{m|n}(\Lambda)$  of $\mathfrak{gl}_{m|n}$ decomposes into a multiplicity-free direct sum of Kac modules of $\mathfrak{gl}_{m|n-1}$:
\[ 
 	K^{m|n}(\Lambda)\cong \bigoplus_{\Lambda \downarrow \Lambda'}K^{m|n-1}(\Lambda').
 \]  
\end{proposition}
\begin{proof}
Let  $\Lambda=(\lambda, \omega)\in \mathfrak{h}^*_{m|n}$ be $\mathfrak{g}_0$-dominant integral, where  $\lambda=(\lambda_1, \dots, \lambda_m)$ and $\omega=(\omega_1, \dots, \omega_{n})$. Let $L_0(\Lambda)\cong L^m(\lambda)\otimes L^n(\omega)$ be a simple $\mathfrak{g}_0$-module equipped  with a trivial $\mathfrak{g}_1$-action. Recall from \eqref{eq: isoK} that $K^{m|n}(\Lambda)={\rm U}(\mathfrak{g})\otimes_{{\rm U}(\mathfrak{g}_0\oplus \mathfrak{g}_1)} L_0(\Lambda)= {\rm U}(\mathfrak{g}_{-1})\otimes L_0(\Lambda)$ can be viewed as a tensor product of two $\mathfrak{g}_0$-modules. It follows from \eqref{eq: g1dec} that, upon restriction to $\mathfrak{g}'_0$, the Kac module $K^{m|n}(\Lambda)$ admits the following factorisation:
\begin{align*}
  K^{m|n}(\Lambda)|_{\mathfrak{g}'_0}& \cong {\rm U}(\mathfrak{g}'_{-1})\otimes \big({\rm U}(\mathfrak{m}'_{-1})\otimes L^m(\lambda) \otimes L^{n}(\omega)\big )|_{\mathfrak{g}'_0}\\
 & ={\rm U}(\mathfrak{g}')\otimes_{{\rm U}(\mathfrak{g}'_0\oplus \mathfrak{g}'_1)} \big( {\rm U}(\mathfrak{m}'_{-1})\otimes L^m(\lambda) \otimes L^{n}(\omega)\big),
\end{align*} 
 where in the last equality $\mathfrak{g}'_1$, as a subalgebra of $\mathfrak{g}_1$, acts trivially on both  ${\rm U}(\mathfrak{m}'_{-1})$ and $L^m(\lambda) \otimes L^{n}(\omega)$.

Consider the $\mathfrak{gl}_m$-module ${\rm U}(\mathfrak{m}'_{-1})\otimes L_{\lambda}^m$. Note that
 ${\rm U}(\mathfrak{m}'_{-1})$ is isomorphic to the exterior algebra  $\bigw((\mathbb{C}^m)^*)= \bigoplus_{k=0}^m \bigw^{k}((\mathbb{C}^m)^*)$ as a $\mathfrak{gl}_m$-module.  
For each $k=0,1,\dots, m$, using  Pieri's rule from \thmref{thm: Pieri},  we obtain
\[ \bigw^{k}( (\mathbb{C}^m)^*)\otimes  L^m(\lambda) \cong \bigoplus_{\lambda\rightarrow \lambda',\ |\lambda-\lambda'|=k} L^m(\lambda'),  \]
Hence as $\mathfrak{gl}_m$-modules, 			
\[{\rm U}(\mathfrak{m}'_{-1})\otimes L^m(\lambda) \cong \bigoplus_{\lambda\rightarrow\lambda'} L^m(\lambda'). \]

On the other hand, by \thmref{thm: glbr}, restricting to  $\mathfrak{gl}_{n-1}$, we have 
\[ L^{n}(\omega)|_{\mathfrak{gl}_{n-1}} \cong \bigoplus_{\omega \downarrow \omega'} L^{n-1}(\omega'), \]
where $\omega \downarrow \omega'$ is defined by \eqref{eq: ladown}. Therefore, we obtain
\[ 
K^{m|n}(\Lambda)\cong {\rm U}(\mathfrak{g}'_{-1}) \otimes \bigoplus_{\lambda \rightarrow \lambda',\  \omega\downarrow \omega'} L^m(\lambda') \otimes L^{n-1}(\omega') \cong \bigoplus_{\Lambda\downarrow \Lambda'} K^{m|n-1}(\Lambda'),  
 \]
where $K^{m|n-1}(\Lambda')= {\rm U}(\mathfrak{g}'_{-1})\otimes (L^m(\lambda')\otimes L^{n-1}(\omega'))={\rm U}(\mathfrak{g}')\otimes_{{\rm U}(\mathfrak{g}'_0\oplus \mathfrak{g}'_1)} (L^m(\lambda')\otimes L^{n-1}(\omega'))$ is the Kac module with highest weight $\Lambda'=(\lambda', \omega')$ for all $\lambda \rightarrow \lambda'$ and $\omega\downarrow \omega'$.
\end{proof}

\begin{remark}\label{rmk: sp}
For $\mathfrak{g}= \mathfrak{gl}_{m|1}$, the Kac module $K^{m|1}(\Lambda)$, where $\Lambda=(\lambda, \omega)\in P^{m|1}_+$, admits a multiplicity-free decomposition into simple $\mathfrak{gl}_m$-modules: 
\[ K^{m|1}(\Lambda)\cong \bigoplus_{\Lambda \downarrow \Lambda'} L^{m}(\Lambda'). \]
To see this, note that  $K^{m|1}(\Lambda)|_{\mathfrak{gl}_m}\cong {\rm U}(\mathfrak{\mathfrak{g}_{-1}})\otimes L^{m}(\lambda)\cong \bigw((\mathbb{C}^m)^*)\otimes L^m(\lambda)$ as $\mathfrak{gl}_m$-modules. The decomposition then follows from Pieri's rules.  
\end{remark}

Recall that the Kac module is simple if and only if it is typical \cite{Kac77b}. We apply \propref{prop: Kacdec}
to typical unitary Kac modules. 
\begin{proposition}\label{prop: tybra}
  Let  $L^{m|n}(\Lambda)$ be a typical unitary  $\mathfrak{gl}_{m|n}$-module with highest weight $\Lambda\in D^{m|n}_{+}$. Then $L^{m|n}(\Lambda)$ admits a multiplicity-free decomposition into typical unitary   $\mathfrak{gl}_{m|n-1}$-modules  :
\[ L^{m|n}(\Lambda) \cong \bigoplus_{\Lambda\downarrow \Lambda'} L^{m|n-1}(\Lambda'). \]
\end{proposition}

\begin{proof}
Since $\Lambda$ is typical, we have $L^{m|n}(\Lambda)\cong K^{m|n}(\Lambda)$. If $n=1$, the decomposition follows directly from \rmkref{rmk: sp}. If $n>1$, we apply \propref{prop: Kacdec} to decompose $L^{m|n}(\Lambda)$ into Kac modules of $\mathfrak{gl}_{m|n-1}$. It remains to prove that each $\Lambda'$ appearing in the decomposition is typical and unitary, and therefore each $K^{m|n-1}(\Lambda')\cong L^{m|n-1}(\Lambda')$ is simple. 

Since $\Lambda\in D^{m|n}_+$ is typical and unitary, by \thmref{thm: unitary}, we have 
 \[ (\Lambda +\rho_{m|n}, \epsilon_m -\delta_{n})= \lambda_m +\omega_{n}-n+1>0. \] 
 As $\Lambda \downarrow \Lambda'$, we have 
\[ \lambda_m \geq \lambda'_m \geq \lambda_m-1, \quad \omega_{n-1} \geq \omega'_{n-1} \geq \omega_{n}.  \]
Hence we obtain 
\[ (\Lambda'+ \rho_{m|n-1}, \epsilon_m-\delta_{n-1})=\lambda'_m+\omega'_{n-1} -n+2\geq \lambda_m+\omega_{n}-n+1>0.   \]
This, combined with \thmref{thm: unitary} and \rmkref{rmk: typ}, implies that  $\Lambda'$ is typical and  unitary.
\end{proof}

\begin{remark}
  In general, if $\Lambda\in P^{m|n}_+$ is typical, each $\Lambda'$ appearing in the decomposition given by \propref{prop: Kacdec}  is not necessarily typical. Therefore,  these Kac modules $K^{m|n-1}(\Lambda')$  are not necessarily simple modules but rather indecomposable modules.  
\end{remark}

\subsection{Proof for the atypical case}\label{sec: atypical}
For any $s\in \mathbb{R}$, let $\mathbb{C}_s$ be the 1-dimensional $\mathfrak{gl}_{m|n}$-module with weight $(s, \dots, s, \allowbreak -s,  \dots, -s)$, where $s$ and $-s$ appear $m$ and $n$ times, respectively. Clearly, $\mathbb{C}_s$ is unitary. For any $\Lambda\in D^{m|n}_+$,   the tensor product $ L(\Lambda^{(s)}):=L(\Lambda)\otimes \mathbb{C}_{s}$ 
is a simple $\mathfrak{gl}_{m|n}$-module with  highest weight
\[ \Lambda^{(s)}:=(\lambda_1+s, \dots, \lambda_m+ s, \omega_1-s, \dots,  \omega_{n}-s).  \]

\begin{lemma}\label{lem: atypoly}
Let $\Lambda=(\lambda_1, \dots, \lambda_m, \omega_1, \dots,\omega_{n})\in D_+^{m|n}$ be an atypical unitary highest weight. Then $L(\Lambda^{(\omega_{n})})$ appears as a simple unitary $\mathfrak{gl}_{m|n}$-submodule of $S(\mathbb{C}^d\otimes \mathbb{C}^{m|n})$ for some positive integer $d$. 
\end{lemma}

\begin{proof}
Since $\Lambda$ is atypical and unitary, by \thmref{thm: unitary},  there exists $\mu \in \{1,2,\dots, n\}$ such that $(\Lambda+\rho, \epsilon_m-\delta_{\mu})= (\Lambda, \delta_{\mu}-\delta_{n})=0$.   Then  we have $\lambda_m+ \omega_{\mu}=\mu-1$ and  $\omega_{\mu}= \omega_{\mu+1}=\cdots= \omega_{n}.$ It follows that 
\begin{equation}\label{eq: modLa}
  \Lambda^{(\omega_{n})}=(\lambda_1+\omega_{n}, \dots, \lambda_m+ \omega_{n}, \omega_1-\omega_{n}, \dots,  \omega_{\mu-1}-\omega_{n}, 0, \dots, 0), 
\end{equation}
whose weight components are all integers. 

We set $d= m+ \omega_1- \omega_{n}$. Applying \thmref{thm: Howe} and \thmref{thm: polyunitary}, $S(\mathbb{C}^d\otimes \mathbb{C}^{m|n})$ decomposes into simple unitary $\mathfrak{gl}_d \oplus \mathfrak{gl}_{m|n}$-modules $L^d(\lambda)\otimes L^{m|n}(\lambda^{\natural}) $ for all $\lambda\in \mathcal{P}_d \cap \mathcal{P}_{m|n}$. It remains  to show that $\Lambda^{(\omega_{n})}= \sigma^{\natural}$ for some partition $\sigma \in \mathcal{P}_d \cap \mathcal{P}_{m|n}$. We define $\sigma=(\sigma_1, \dots, \sigma_d)$ as follows: $\sigma_i= \lambda_i +\omega_{n}$ for $1\leq i\leq m$, and 
\[(\sigma_{m+1},\dots ,\sigma_d)= (\omega_1-\omega_{n}, \dots, \omega_{\mu-1}-\omega_{n})^c,  \]
where the right hand side is a conjugate partition. Hence we obtain two  non-increasing integer sequences $\sigma_1\geq \sigma_2\geq \cdots \geq \sigma_m$ and $\sigma_{m+1}\geq \cdots\geq \sigma_d\geq 0$. By the definition of the conjugate partition, we have $\sigma_{m+1}\leq \mu-1$. On the other hand,  $\sigma_m= \lambda_m+\omega_{n}=\lambda_m+\omega_{\mu}= \mu-1$. It follows that $\sigma_m\geq \sigma_{m+1}$ and   $\sigma_{m+1}\leq \mu-1\leq n$. Therefore, $\sigma$ is an $(m, n)$-hook partition, and  it is easily verified that $\sigma^{\natural}=\Lambda^{(\omega_{n})}$. 
\end{proof}

\begin{remark}
  It was originally shown in \cite{GZ90} that  $L(\Lambda^{(\omega_n)})$  appears as a submodule of a tensor power of the natural module  $\mathbb{C}^{m|n}$  of  $\mathfrak{gl}_{m|n}$, and is therefore a polynomial module. In \lemref{lem: atypoly} we prove this fact directly using the Howe duality. 
\end{remark}

The following branching rule for polynomial modules is well-known \cite{BR87,SV10}. We provide a quick proof using the Howe duality; see also \cite{CPT15, LZ24}.  
\begin{proposition}\label{prop: hookbra}
Let $\lambda$ be an $(m,n)$-hook partition. The finite dimensional simple $\mathfrak{gl}_{m|n}$-module $L^{m|n}(\lambda^{\natural})$ admits the following multiplicity-free decomposition under the action of $\mathfrak{gl}_{m|n-1}$:
\[ L^{m|n}(\lambda^{\natural})= \bigoplus_{\lambda \rightarrow \lambda'} L^{m|n-1}(\lambda'^{\, \natural}),  \]
where the sum is taken over all $(m,n-1)$-hook partitions $\lambda'$ such that $\lambda \rightarrow \lambda'$, as defined in \eqref{eq: lasim}. 
\end{proposition}
\begin{proof}
Assume that $\lambda$ has length $d$. By the  $(\mathfrak{gl}_d, \mathfrak{gl}_{m|n})$-Howe duality from  \thmref{thm: Howe},  $L^d(\lambda)\otimes L^{m|n}(\lambda^{\natural})$ appears as a simple $\mathfrak{gl}_d\otimes \mathfrak{gl}_{m|n}$-submodule of $S(\mathbb{C}^d\otimes \mathbb{C}^{m|n})$. Restricting to $\mathfrak{gl}_d\oplus (\mathfrak{gl}_{m|n-1}\oplus \mathfrak{gl}_1)$, we have the isomorphism  $\mathbb{C}^d \otimes \mathbb{C}^{m|n}\cong \mathbb{C}^d \otimes \mathbb{C}^{m|n-1} \oplus \mathbb{C}^d \otimes \mathbb{C}^{0|1}$. This yields 
\begin{align*}
S(\mathbb{C}^d \otimes \mathbb{C}^{m|n})&\cong S(\mathbb{C}^d \otimes \mathbb{C}^{m|n-1} )\otimes S(\mathbb{C}^d \otimes \mathbb{C}^{0|1})\\
&\cong \bigoplus_{\lambda' \in \mathcal{P}_d \cap \mathcal{P}_{m|n-1}}L^{d}(\lambda')\otimes L^{m|n-1}(\lambda'^{\, \natural})\otimes \bigw(\mathbb{C}^d)   \\ 
&\cong \bigoplus_{\lambda' \in \mathcal{P}_d \cap \mathcal{P}_{m|n-1}} \bigoplus_{\sigma \rightarrow \lambda'} L^d(\sigma) \otimes L^{m|n-1}(\lambda'^{\, \natural}),
\end{align*}
where the last isomorphism follows from  Pieri's rule. By comparing this decomposition with the original appearance of $L^d(\lambda) \otimes L^{m|n}(\lambda^{\natural})$ in $S(\mathbb{C}^d \otimes \mathbb{C}^{m|n})$, we obtain the desired decomposition of $L^{m|n}(\lambda^{\natural})$ as a $\mathfrak{gl}_{m|n-1}$-module.  
\end{proof}
\begin{remark}\label{rmk: brawt}
  In terms of Young diagrams, the direct sum in \propref{prop: hookbra} is taken over all $(m, n-1)$-hook partitions $\lambda^{\prime}$ such that the skew diagram $\lambda / \lambda^{\prime}$ forms a vertical strip, i.e., each row of $\lambda / \lambda^{\prime}$ contains at most one box. For more details on skew diagrams, the reader can refer to \cite[Section 9.3.5]{GW09} 
 
 In terms of highest weights, the branching rule for $L^{m|n}(\lambda^{\natural})$ can be described as follows. If the length $d$ of $\lambda$ is less than $m$, then  by definition \eqref{eq: lanat} we have $\lambda^{\natural}=(\lambda_1, \dots, \lambda_d, 0, \dots ,0)$. Thus,  the direct sum in \propref{prop: hookbra} is taken over all highest weights $\lambda'^{\natural}\in D_{+}^{m|n-1}$ such that $\lambda^{\natural}_i - \lambda'^{\natural}_i\in \{0,1\}$ for all $1\leq i\leq d$.  If $d\geq m$, then using the bijection between  the Young diagrams of $(m,n)$-hook partitions $\lambda$ and highest weights $\lambda^\natural$,  the branching condition $\lambda\rightarrow \lambda'$ is equivalent to $\lambda^{\natural}_i - \lambda'^{\natural}_i\in \{0,1\}$ for all $1\leq i\leq m$ and $\lambda^{\natural}_{m+j}\geq \lambda'^{\natural}_{m+j}\geq \lambda^{\natural}_{m+j+1}$ for all $1\leq j\leq n-1$. Therefore, in both cases we have $\lambda^{\natural}\downarrow \lambda'^{\natural}$.
\end{remark}

\begin{proposition}\label{prop: atybra}
  Let  $L^{m|n}(\Lambda)$ be an atypical unitary  $\mathfrak{gl}_{m|n}$-module with highest weight $\Lambda\in D^{m|n}_{+}$. Then $L^{m|n}(\Lambda)$ admits a multiplicity-free decomposition into  unitary  simple  $\mathfrak{gl}_{m|n-1}$-modules:
  \[ L^{m|n}(\Lambda) \cong \bigoplus_{\Lambda\downarrow \Lambda'} L^{m|n-1}(\Lambda'). \]
\end{proposition}
\begin{proof}
   
Note that $L^{m|n}(\Lambda)\cong L^{m|n}(\Lambda^{(\omega_n)})\otimes \mathbb{C}_{-\omega_n}$ as a $\mathfrak{gl}_{m|n}$-module. We shall write   $\Lambda=(\lambda_1, \dots, \lambda_m, \omega_1, \dots, \omega_n)$ and $\Lambda'=(\lambda'_1, \dots, \lambda'_m, \omega'_1, \dots, \omega'_{n-1})$.    Since $\Lambda$ is atypical and unitary, by \thmref{thm: unitary}, there exists $1\leq \mu\leq n$ such that $(\Lambda+\rho, \epsilon_m-\delta_{\mu})=(\Lambda, \delta_1-\delta_{\mu})=0$. We consider two cases for $\mu$.

Case 1:  $\mu=1$. We have $\lambda_m+\omega_1=0$ and $\omega_1=\omega_2=\dots=\omega_n$. 
Hence, 
\[
  \Lambda^{(\omega_n)}=(\Lambda^{(\omega_n)}_1, \dots , \Lambda^{(\omega_n)}_{m+n})=  (\lambda_1+\omega_n, \dots, \lambda_{m-1}+\omega_n, 0, \dots, 0).
\] 
As shown in the poof of \lemref{lem: atypoly},  one can construct an $(m,n)$-hook partition $\lambda$ associated to $\Lambda^{(\omega_n)}$ of length $d= m-1$. Hence,  by \propref{prop: hookbra}, $L^{m|n}(\Lambda^{(\omega_n)})$ decomposes into a multiplicity-free direct sum of  simple unitary $\mathfrak{gl}_{m|n-1}$ modules $L^{m|n-1}(\Lambda'^{(\omega_n)})$, where  $\Lambda'^{(\omega_n)}=(\Lambda'^{(\omega_n)}_1, \dots , \Lambda'^{(\omega_n)}_{m-1}, 0\dots, 0) \in D^{m|n-1}_+$ with $\Lambda^{(\omega_n)}_i-\Lambda'^{(\omega_n)}_i\in \{0,1\}$ for all $1\leq i\leq m-1$. Tensoring with $\mathbb{C}_{-\omega_n}$, we obtain the desired decomposition of $L^{m|n}(\Lambda)$ into simple unitary $\mathfrak{gl}_{m|n-1}$-modules $L^{m|n-1}(\Lambda')$, where $\lambda_m= \lambda_m'=-\omega_n$ and $\lambda_i-\lambda'_i=\Lambda^{(\omega_n)}_i-\Lambda'^{(\omega_n)}_i \in\{0,1\}$ for all  $1\leq i\leq m-1$, as well as $\omega_1'=\dots =\omega'_{n-1}=\omega_n$. These satisfy conditions (C1) and (C2) of the interlacing relation $\Lambda \downarrow \Lambda'$.     

Case 2: $\mu>1$. Here,  $\lambda_m+\omega_n=\mu-1>0$. Hence   $\Lambda^{(\omega_n)}$ is of  form \eqref{eq: modLa} with $\Lambda^{(\omega_n)}_m= \lambda_m+\omega_n>0$. Following  a similar process as above, we obtain the desired decomposition of $L^{m|n}(\Lambda)$ into simple unitary $\mathfrak{gl}_{m|n-1}$-modules $L^{m|n-1}(\Lambda')$, where $\lambda_i-\lambda'_i\in \{0,1\}$ for all $1\leq i\leq m$ and $ \omega_{\mu} \geq \omega'_{\mu}\geq  \omega_{\mu+1}$ and $\omega_{\mu}- \omega'_{\mu}\in \mathbb{Z}_+$ for all  $1\leq \mu \leq n-1$, as specified by the  interlacing relation $\Lambda \downarrow \Lambda'$. 
\end{proof}

\begin{remark}
Unlike  in \propref{prop: tybra}, the modules $L^{m|n-1}(\Lambda')$ appearing  in the decomposition given by \propref{prop: atybra} are not necessarily atypical, although $\Lambda$ is atypical.  
\end{remark}

Now we summarise the proof of \thmref{thm: main}. 
\begin{proof}[Proof of \thmref{thm: main}]
  By \thmref{thm: unitary} and \rmkref{rmk: typ}, the highest weight $\Lambda$ of $L^{m|n}(\Lambda)$ is either typical or atypical. In the typical case, the desired decomposition follows from \propref{prop: tybra}, whereas in the atypical case, it follows from \propref{prop: atybra}.
\end{proof}

\section{Type 2 unitary branching rules}\label{sec: type2uni}
As mentioned in the Introduction, there are two types of unitary $\mathfrak{gl}_{m|n}$-modules \cite{SNR77, GZ90}. Using the branching rules of type 1 unitary modules, we will derive the corresponding rules for type 2 unitary modules.

\subsection{Type 2 unitary modules}

Following \cite{SNR77, GZ90}, a simple $\mathfrak{gl}_{m|n}$-module $L(\Upsilon)$ is said to be \emph{type 2 unitary} if there exists a positive definite Hermitian form $\langle-, -\rangle$ on $L(\Upsilon)$ such that for all $v,w\in  L(\Upsilon)$, 
\begin{equation}\label{eq: Her2}
  \langle E_{ij}v, w\rangle= (-1)^{[E_{ij}]} \langle v, E_{ji} w\rangle, \quad 1\leq i,j \leq m+n.
\end{equation}
 Type 2 unitary modules incorporate signs in their definition, compared to definition \eqref{eq: Her} of (type 1) unitary modules. Indeed, they are related by duality, as stated below.

\begin{proposition}\cite[Proposition 1]{GZ90}\label{prop: type2}
Let $\Lambda\in D^{m|n}_+$. The simple $\mathfrak{gl}_{m|n}$-module $L(\Lambda)$ is (type 1) unitary if and only if its dual $L(\Lambda)^*$ is type 2 unitary. 
\end{proposition}

If  $\Lambda^-$ is the lowest weight of a simple $\mathfrak{gl}_{m|n}$-module $L(\Lambda)$, then the highest weight of the dual $L(\Lambda)^*$ is $\Lambda^*= -\Lambda^-$.  For typical unitary module $L(\Lambda)$, it follows from the  Kac module construction  that the  highest weight of the dual is  
\begin{equation}\label{eq: dualLa}
  \Lambda^*= -w_0(\Lambda -\sum_{i=1}^{m}\sum_{\mu=1}^n(\epsilon_i-\delta_\mu))= (n-\lambda_m, \dots, n-\lambda_1, -\omega_n-m, \dots, -\omega_1-m),
\end{equation}
where $\Lambda=(\lambda_1, \dots, \lambda_m, \omega_1, \dots, \omega_n)$ and $w_0$ is the longest element of the Weyl group of $\mathfrak{gl}_{m|n}$. For atypical unitary module $L(\Lambda)$, an explicit formula for $\Lambda^{\ast}$ was  determined  in \cite{GZ90}; we will give an alternative formula below.  The formulas for $\Lambda^*$, combined with \thmref{thm: unitary}, yield the following  classification of type 2 unitary modules. 

\begin{theorem}\cite[Theorem 4]{GZ90}\label{thm: type2unitary}
  Let $\Upsilon\in D_+^{m|n}$. The finite-dimensional simple  $\mathfrak{gl}_{m|n}$-module $L(\Upsilon)$ is type 2 unitary if and only if 
  \begin{enumerate}
  \item $(\Upsilon+ \rho, \epsilon_1-\delta_1)<0 $; or 
  \item there exists $k \in \{1, 2, \dots, m\}$ such that 
  \[(\Upsilon+ \rho, \epsilon_k-\delta_{1}) = (\Upsilon, \epsilon_1-\epsilon_k)=0. \]
  \end{enumerate}
\end{theorem}

\thmref{thm: type2unitary} implies that taking the dual preserves both typicality and atypicality of unitary simple $\mathfrak{gl}_{m|n}$-modules. More precisely,   if $\Lambda$ is a typical unitary highest weight satisfying $(\Lambda+\rho, \epsilon_m-\delta_n)>0$, then the dual highest weight $\Lambda^*$ is typical type 2 unitary satisfying $(\Lambda^*+ \rho, \epsilon_1-\delta_1)<0$; and vice versa. This holds similarly for atypical unitary modules.

Recall from \lemref{lem: atypoly} that every atypical unitary module can be realised as a polynomial module up to a twist by a 1-dimensional module.  It is thus natural to describe highest weights of the dual atypical unitary modules in terms of partitions. Let $\lambda=(\lambda_1, \dots, \lambda_d)$ be an $(m,n)$-hook partition, and let $L(\lambda^{\natural})$ be the polynomial module with the highest weight $\lambda^{\natural}$ associated to $\lambda$. Recall that $\lambda^{\natural}$ is given by  formula \eqref{eq: lanat}.  The lowest weight of $L(\lambda^{\natural})$  is 
$\lambda^{-}=(\langle \lambda_m-n\rangle, \dots, \langle \lambda_1-n\rangle, \lambda^c_n, \dots, \lambda^c_1)$; refer to \cite[\S 2.4]{CW12} or \cite[Proposition A.2]{Zha20}.   Therefore, the highest weight of the dual $L(\lambda^{\natural})^*$ is $\lambda^*=-\lambda^{-}$. This establishes a bijection between $(m,n)$-hook partitions $\lambda$ and dual highest weights $\lambda^*$ (with respect to the standard Borel subalgebra).  

Taking dual on both sides of the decomposition of $L(\lambda^{\natural})$ from \propref{prop: hookbra}, we obtain the branching rule for its dual module
\begin{equation}\label{eq: dualdec}
  L^{m|n}(\lambda^{\natural})^*= \bigoplus_{\lambda \rightarrow \lambda'} L^{m|n-1}(\lambda'^{\, \natural})^*. 
\end{equation}
Therefore, the decomposition of the dual obeys the same branching rule in terms of partitions. Alternatively, this can be described in terms of highest weights as follows. Let $r\in \{0, 1, \dots, m\}$ be the integer such that the first $m$ parts of $\lambda=(\lambda_1, \dots, \lambda_d)$ satisfy 
\begin{equation}\label{eq: atyla}
  \lambda_1\geq \lambda_2\geq \dots \geq \lambda_r>n\geq \lambda_{r+1}\geq \dots \geq \lambda_m,
\end{equation}
where  $\lambda_k:=0$ if $k>d$. Note that if $r=m$, then $\lambda^{\natural}$ is typical.  Then the highest weight of $L^{m|n}(\lambda^{\natural})^*$ is 
\[ \lambda^*=-\lambda^{-}= (0, \dots, 0, n-\lambda_r, \dots, n-\lambda_1, -\lambda^c_n, \dots, -\lambda^c_1),\] 
where the first $m-r$ entries are all $0$. Since $\lambda \rightarrow \lambda'$, i.e., $\lambda_i-\lambda'_i\in \{0,1\}$ for all $1\leq i\leq d$, the highest weight $\lambda^{\prime *}$ of $L^{m|n-1}(\lambda'^{\, \natural})^*$ is related to $\lambda^*$ by
\begin{equation}\label{eq: polyint}
  \begin{aligned}
    &\lambda^*_i= \lambda^{\prime *}_i=0, \quad  1\leq i\leq m-r, \\
    & \lambda^*_i- \lambda^{\prime *}_i\in \{0,1\}, \quad m-r+1 \leq i\leq m, \\
    & \lambda^*_{m+j}\geq \lambda^{\prime *}_{m+j}\geq \lambda^*_{m+j+1}, \quad 1\leq j\leq n-1.  
  \end{aligned}
\end{equation}

\subsection{Type 2 unitary branching rules}
We proceed  to derive the branching rules of type 2 unitary modules by employing the duality betweens type 1 and type 2 unitary modules. 

Let $\Upsilon=(\lambda_1,\dots, \lambda_m, \omega_1, \dots, \omega_{n}) \in P^{m|n}_+$   and $\Upsilon'= (\lambda'_1,\dots, \lambda'_m, \omega'_1, \dots, \omega'_{n-1})\in  P^{m|n-1}_+$. We shall write $\Upsilon \Downarrow \Upsilon'$ if the following type 2 interlacing  conditions hold: 
\begin{enumerate}
\item [(C1$'$)]If there exists a maximal  $k \in \{1, 2, \dots, m\}$ such that $(\Upsilon+\rho, \epsilon_k-\delta_1) = 0$, then $\lambda_i = \lambda^{\prime}_i$ for all  $1 \leq i \leq k$ , and $\lambda_i - \lambda^{\prime}_i \in \{0, 1\}$ for $k < i \leq m$. If no such  $k$ exists, then $\lambda_i - \lambda^{\prime}_i \in \{0, 1\}$ for all $1 \leq i \leq m$. 
\item [(C2)] $ \omega_{\mu} \geq \omega'_{\mu}\geq  \omega_{\mu+1}$ and $\omega_{\mu}- \omega'_{\mu}\in \mathbb{Z}_+$ for all  $1\leq \mu \leq n-1$. 
\end{enumerate}
We note that if $L^{m|n}(\Upsilon)$ is type 2 unitary and atypical, then in (C1$'$), the equality $(\Upsilon+\rho, \epsilon_k-\delta_1) = 0$ implies that $(\Upsilon, \epsilon_1-\epsilon_k)=0$ by \thmref{thm: type2unitary}.

\begin{lemma}\label{lem: dualdec}
Let $L^{m|n}(\Lambda)$ be a type 1 unitary simple $\mathfrak{gl}_{m|n}$-module. Then the dual module $L^{m|n}(\Lambda)^*$ admits the following multiplicity-free decomposition under the action of $\mathfrak{gl}_{m|n-1}$:   
\[ L^{m|n}(\Lambda)^*\cong \bigoplus_{\Lambda^* \Downarrow \Lambda^{'*}} L^{m|n-1}(\Lambda')^*.   \]
\end{lemma}
\begin{proof} 
 By taking dual on both sides of the decomposition in \thmref{thm: main}, we obtain 
 \[ L^{m|n}(\Lambda)^*\cong \bigoplus_{\Lambda \downarrow \Lambda'} L^{m|n-1}(\Lambda')^*.   \]
It remains to show that  $\Lambda^*\Downarrow \Lambda^{'*}$ for every pair $\Lambda \downarrow \Lambda'$. 

If $\Lambda$ is typical, then by \propref{prop: tybra}  $\Lambda'$ is also typical. It follows from formula \eqref{eq: dualLa} that $\Lambda^*\Downarrow \Lambda^{'*}$ for any $\Lambda'$ such that $\Lambda\downarrow \Lambda'$. 

If $\Lambda=(\lambda_1, \dots, \lambda_n, \omega_1,\dots, \omega_n)$ is atypical, then there exists $1\leq \mu\leq n$ such that $(\Lambda+\rho, \epsilon_m-\delta_{\mu})=(\Lambda, \delta_{\mu}-\delta_n)=0$. By \lemref{lem: atypoly},  $L(\Lambda^{(\omega_n)})=L(\Lambda)\otimes \mathbb{C}_{\omega_{n}}$ is a polynomial module with highest weight 
\[  \Lambda^{(\omega_{n})}=(\lambda_1+\omega_{n}, \dots, \lambda_m+ \omega_{n}, \omega_1-\omega_{n}, \dots,  \omega_{\mu-1}-\omega_{n}, 0, \dots, 0),\]
where $\lambda_m+ \omega_{n}=\mu-1<n$. It suffices to prove that  $(\Lambda^{(\omega_n)})^*\Downarrow (\Lambda^{'(\omega_n)})^{*}$. 

As in the proof of \lemref{lem: atypoly}, there exists an  $(m,n)$-hook partition $\sigma=(\sigma_1, \dots, \sigma_d)$ such that $\sigma^{\natural}= \Lambda^{(\omega_n)}$, with $d= m+ \omega_1-\omega_n$. Recall that the components are given by  $\sigma_i=\lambda_i+ \omega_n$ for $1\leq i\leq m$ and $(\sigma_{m+1},\dots ,\sigma_d)= (\omega_1-\omega_{n}, \dots, \omega_{\mu-1}-\omega_{n})^c$.  As described in \eqref{eq: dualdec}, the module  $L^{m|n}(\sigma^{\natural})^*$ decomposes into a multiplicity-free direct sum of simple modules $L^{m|n-1}(\sigma^{\, \prime \natural})^*$ over all $\sigma'$ such that  $\sigma\rightarrow \sigma'$ (which is equivalent to $\sigma^{\natural}\downarrow \sigma'^{\natural}$). Let $r\in \{0,1,\dots, m-1\}$ be the integer for which \eqref{eq: atyla} is satisfied for $\sigma$. Then we have 
\[(\sigma^{*}+ \rho, \epsilon_{m-r}-\delta_1)= \sigma^{*}_{m-r}+ \sigma^{*}_{m+1}+ r =0+ (-r)+r=0.\]
The highest weights $\sigma^*$ and $\sigma^{\prime *}$ of the dual modules $L^{m|n}(\sigma^{\natural})^*$ and  $L^{m|n-1}(\sigma^{\, \prime \natural})^*$ are related as in \eqnref{eq: polyint}. Thus, conditions (C1$'$) and (C2) are satisfied, so $\sigma^{*}\Downarrow \sigma^{\prime *}$.  
\end{proof}

\begin{theorem}\label{thm: maintype2}
  Let  $L^{m|n}(\Upsilon)$ be a finite-dimensional type 2 unitary simple $\mathfrak{gl}_{m|n}$-module with highest weight $\Upsilon\in D^{m|n}_+$. Then    $L^{m|n}(\Lambda)$ decomposes into a  multiplicity-free direct sum of  type 2 unitary simple $\mathfrak{gl}_{m|n-1}$-modules: 
  \[
  L^{m|n}(\Upsilon)\cong \bigoplus_{\Upsilon \Downarrow \Upsilon'} L^{m|n-1}(\Upsilon'). 
  \]
\end{theorem}
\begin{proof}
By \propref{prop: type2}, the dual module $L^{m|n}(\Upsilon)^*$ is a type 1 unitary module  with highest weight $\Upsilon^*$. Note that the double dual induces an isomorphism $L^{m|n}(\Upsilon)^{**}\cong L^{m|n}(\Upsilon)$ as $\mathfrak{gl}_{m|n}$-modules with the same highest weight $\Upsilon$. Applying \lemref{lem: dualdec} to $L^{m|n}(\Upsilon)^*$, we obtain the desired decomposition.
\end{proof}

\subsection{Some remarks}
We conclude this section with some remarks.

If $L(\Lambda)$ is an atypical unitary simple $\mathfrak{gl}_{m|n}$-module, then by \lemref{lem: atypoly}, it is a polynomial module corresponding to an $(m,n)$-hook partition $\lambda$, up to a twist by a one-dimensional module. In terms of Young diagrams, the interlacing condition $\Lambda \downarrow \Lambda^{\prime}$ is equivalent to the skew diagram $\lambda/\lambda^{\prime}$ being a vertical strip. By \eqref{eq: dualdec}, this Young diagrammatic description also holds for atypical type 2 unitary simple $\mathfrak{gl}_{m|n}$-modules. Therefore, the Young diagrammatic description of the branching rules for atypical unitary $\mathfrak{gl}_{m|n}$-modules remains unchanged when taking the dual. For the typical unitary case, the algebraic description of the branching rules remains unchanged under duality, as seen from conditions (C1), (C1$'$) and (C2).

Unitary branching rules naturally extend to quantum groups ${\rm U}_q(\mathfrak{gl}_{m|n})$ for any positive real number $q$. As shown in \cite{GS95}, all finite-dimensional unitary simple $\mathfrak{gl}_{m|n}$-modules quantise to yield unitary simple ${\rm U}_q(\mathfrak{gl}_{m|n})$-modules. Moreover, the classification of types 1 and 2 unitary simple ${\rm U}_q(\mathfrak{gl}_{m|n})$-modules in terms of highest weights is the same as in the $\mathfrak{gl}_{m|n}$ case. Consequently, the branching rules given in \thmref{thm: main} and \thmref{thm: maintype2} apply directly to  simple unitary ${\rm U}_q(\mathfrak{gl}_{m|n})$-modules.

\vspace*{0.2cm}
\noindent {\bf Acknowledgement.}
We thank Professor Ruibin Zhang for stimulating discussions.


\begin{thebibliography}{9999}

\bibitem[BR87]{BR87}
A. Berele, A. Regev.  \emph{Hook Young diagrams with applications to combinatorics and to representations of Lie superalgebras}. Adv. Math. 64, 118-175, 1987. 

\bibitem[CPT15]{CPT15}
S.~Clark, Y.~N.~Peng, S.~K.~Thamrongpairoj. 
\emph{Super tableaux and a branching rule for the general linear Lie superalgebra.} Linear Multilinear Algebra 63, no. 2, 274-282, 2015.

\bibitem[CLZ04]{CLZ04}
 S. -J. Cheng, N. Lam, R. B. Zhang. \emph{Character formula for infinite-dimensional unitarizable modules of the general linear superalgebra.} J. Algebra 273, 780-805, 2004.

\bibitem[CW01]{CW01}
S.~-J. Cheng, W. Wang. \emph{Howe duality for Lie superalgebras.} Compositio Math. 128, 53-94, 2001.

\bibitem[CW12]{CW12}
S.~-J. Cheng, W. Wang. \emph{Dualities and representations of Lie superalgebras.} Grad. Stud. Math., 144, American Mathematical Society, Providence, RI, 2012, xviii+302 pp.


\bibitem[GW09]{GW09}
R. Goodman and N. Wallach. \emph{Symmetry, Representations, and Invariants}, GTM 255, Springer, New York, 2009.

\bibitem[Gou89]{Gou89}
M.~D.~Gould. \emph{Casimir invariants and infinitesimal characters for semi-simple Lie algebras.}
Rep. Math. Phys.27, no.1, 73-111, 1989.


\bibitem[GIW13]{GIW13}
M.~D.~Gould., P.~S.~Isaac, J. L.~Werry. \emph{Invariants and reduced matrix elements associated with the Lie superalgebra   $\mathfrak{gl}(m|n)$.} 
J. Math. Phys. 54, no. 1, 013505, 34 pp, 2013.

\bibitem[GIW14]{GIW14}
M.~D.~Gould., P. S.~Isaac, J. L.~Werry. \emph{Matrix elements for type 1 unitary irreducible representations of the Lie superalgebra  $\mathfrak{gl}(m|n)$.} J. Math. Phys. 55, no. 1, 011703, 32 pp, 2014. 

\bibitem[GIW15]{GIW15}
M.~D.~Gould., P. S.~Isaac, J. L.~Werry. \emph{Matrix elements and duality for type 2 unitary representations of the Lie superalgebra $\mathfrak{gl}(m|n)$.} J. Math. Phys. 55, no. 1, 011703, 32 
J. Math. Phys. 56, no. 12, 121703, 19 pp, 2015.


\bibitem[GS95]{GS95}
M.~D.~Gould, M.~Scheunert. \emph{Classification of finite-dimensional unitary irreps for  ${\rm U}_q(\mathfrak{gl}_{m|n})$.} J. Math. Phys. 36, no. 1, 435-452, 1995.


\bibitem[GZ90]{GZ90}
M.~D.~ Gould, R.~B.~Zhang. 
\newblock{\em{Classification of all star irreps of $\mathfrak{gl}(m|n)$}.}
\newblock{J. Math. Phys.}, 31, no. 11, 2552-2559, 1990. 

\bibitem[How89]{How89}
R. Howe. \emph{Remarks on classical invariant theory.} Trans. Amer. Math. Soc. 313, 539-570, 1989.


\bibitem[Kac77a]{Kac77a}
V.~G.~Kac. 
\newblock{ \em{Lie superalgebras.}} Advances in Mathematics. 26 (1): 8-96, 1977. 


\bibitem[Kac77b]{Kac77b}
V.~G.~Kac. 
\emph{Representations of classical Lie superalgebras.} Lecture Notes in Mathematics 676, eds. K. Bleuler, H. Petry and A. Reetz, Springer, Berlin, 579-626, 1977.


\bibitem[LZ24]{LZ24}
S.~T.~Lee, R.B.~Zhang. \emph{Branching algebras for the general linear Lie superalgebra.} arXiv: 2403.11393, 2024.

\bibitem[Mo06]{Mo06}
A.~I.~Molev. \emph{Gelfand-Tsetlin bases for classical Lie algebras.}
Handb. Algebr., 4, Elsevier/North-Holland, Amsterdam, 2006, 109-170.

\bibitem[Mo11]{Mo11}
A.~I.~Molev. \emph{Combinatorial bases for covariant representations of the Lie superalgebra $\mathfrak{gl}_{m|n}$.} Bull. Inst. Math. Acad. Sin. (N.S.) 6, no. 4, 415-462, 2011. 


\bibitem[Ser01]{Ser01} A. Sergeev, \emph{An analog of the classical invariant theory for Lie superalgebras.} I, Michigan Math. J. 49, 113-146, 2001.

\bibitem[SNR77]{SNR77}
M. Scheunert, W. Nahm, V. Rittenberg,  \emph{Graded Lie algebras: Generalization of Hermitian representations.} J. Math. Phys. 18 (1977) 146-154.


\bibitem[SV10]{SV10}
N.~I.~Stoilova, J. Van der Jeugt.  
\emph{Gelfand-Zetlin basis and Clebsch-Gordan coefficients for covariant representations of the Lie superalgebra  $\mathfrak{gl}_{m|n}$.} J. Math. Phys. 51, no. 9, 093523, 15 pp, 2010.


\bibitem[Zha20]{Zha20}
Y.~Zhang. \emph{The first and second fundamental theorems of invariant theory for the quantum general linear supergroup.} J. Pure Appl. Algebra 224, no. 11, 106411, 50 pp, 2020. 


\end{thebibliography}
\end{document}